\documentclass[11pt]{article}

\usepackage[latin1]{inputenc}
\usepackage[english]{babel}

\usepackage{amsmath,amsthm,amssymb}
\usepackage{epsfig}
\usepackage{graphics}
\usepackage{graphicx}
\usepackage{grffile}
\usepackage{booktabs}
\usepackage{ctable}
\usepackage{subfigure}
\usepackage{float}
\usepackage{layout}
\usepackage{fancyhdr}
\usepackage{multirow}  
\usepackage{makeidx}   
\usepackage{varioref}
\usepackage{xr}
\usepackage{calc}
\usepackage{eurosym}   

\input amssym.def
\input amssym.tex

\makeindex 

\def \frac#1#2{{#1\over #2}}
\def\1{{\bf 1}}
\def\0{{\bf 0}}
\def\2{{\bf 2}}

\def\x{{\bf x}}

\def\v{{\bf v}}
\def\w{{\bf w}}
\def\b{{\bf b}}
\def\I{{\bf I}}

\def\cfrac#1#2{{#1\over #2}}

\newtheorem{defi}{Definition}[section]
\newtheorem{risu}{Result}[section]

\newtheorem{prop}{Proposition}[section]

\newtheorem{rem}{Remark}[section]
\newtheorem{pr}{Property}[section]

\newcommand{\be}{\begin{enumerate}}
\newcommand{\ee}{\end{enumerate}}
\newcommand{\bi}{\begin{itemize}}
\newcommand{\ei}{\end{itemize}}
\newcommand{\beq}{\begin{equation}}
\newcommand{\eeq}{\end{equation}}

\numberwithin{equation}{section} 

\begin{document}

\baselineskip=14 pt

\title{Bilateral symmetry and modified Pascal triangles\\ in Parsimonious games\footnote{A very preliminary version of the paper has been presented at the 2013 Workshop of the Central European Program in Economic Theory, which took place in Udine (20-21 June) and may be found in CEPET working papers \cite{PrePla13}.}}

\author{Flavio Pressacco\textsuperscript{a}, Giacomo Plazzotta\textsuperscript{b} and Laura Ziani\textsuperscript{c}}

\date{}

\maketitle

\noindent\textsuperscript{a} Dept. of Economics and Statistics D.I.E.S., Udine University, Italy\\ 
\textsuperscript{b} Imperial College London, UK\\
\textsuperscript{c} Dept. of Economics and Statistics D.I.E.S., Udine University, Italy 

\begin{abstract}
We discuss the prominent role played by bilateral symmetry and
modified Pascal triangles in self twin games, a subset of
constant sum homogeneous weighted majority games. We show that
bilateral symmetry of the free representations unequivocally
identifies and characterizes this class of games and that
modified Pascal  triangles describe their cardinality for
combinations of $m$ and $k$, respectively linked through linear
transforms to the key parameters $n$, number of players and
$h$, number of types in the game.

Besides, we derive
the whole set of self twin games in the form of a genealogical
tree obtained through a simple constructive procedure in which
each game of a  generation, corresponding to a given value of
$m$, is able to give birth to one child or two children
(depending on the parity of $m$), self twin games of the next
generation. The breeding rules are, given the parity of $m$, invariant
through generations and quite simple.

\end{abstract}

\renewcommand{\abstractname}{Keywords}
\begin{abstract}
Homogeneous weighted majority games; bilateral
symmetry; modified Pascal triangles; games representations.
\end{abstract}

\renewcommand{\abstractname}{Acknowledgements}
\begin{abstract}
 We wish to thank Michele Giacomini, former student of the 2013 Game Theory
short course at the Scuola Superiore Università di Udine,
for a lot of valuable suggestions and comments on an earlier draft of the paper.
\end{abstract}

\section{Introduction}

Bilateral Symmetry ($BS$) and Modified Pascal Triangles ($MPT$)
are issues that play a significant role in various fields of
sciences. Just to recall some of the most important papers on these subjects let
us quote, for $BS$, the milestone book of Weyl (\cite{We52},
1952) and the papers by Gardner (\cite{Ga71}, 1971), Møller and
Thornill (\cite{MoTh98}, 1998), Finnerty (\cite{Fi03}, 2003),
Song et al (\cite{Song10}, 2010), Palmer (\cite{Pa04}, 2004),
and for $MPT$, Ando (\cite{An88}, 1988), Granville
(\cite{Gra92}, 1992), Barry (\cite{Bar06}, 2006), Trzaska
(\cite{Trz91}, 1991) and Bollinger (\cite{Bo93}, 1993).

In this paper we will see that both issues jointly emerge in a
game theory framework: precisely, as a key to self twin\footnote{In Isbell terminology, self dual games.}
``parsimonious'' games, a class of games introduced by Isbell (\cite{Isb56}, 1956) in
the early stage of game theory,
which at the best of our knowledge did not receive then any
attention.

Parsimonious games (hereafter $P$ games) are the subset of
constant sum homogeneous weighted majority games characterized
by the ``parsimony property'' to have, for any number $n$ of
non dummy players in the game, the smallest number (i.e.
exactly $n$) of minimal winning coalitions.

For such games the incidence matrix $M$ is the binary square
$n\times n$ matrix whose entries are 1 (or 0) if column player
$i$ belongs (does not belong) to the row minimal winning
coalition $j$. A twin (dual) relationship on $P$ games has been
introduced by Isbell (\cite{Isb56}, p. 185) through the
following property of their incidence matrices: the transposed
${M}^{T}$ of the incidence matrix ${M}$ of any $P$ game ${G}$
is the incidence matrix ${\overline{M}}$ of a $P$ game
${\overline{G}}$ called the twin of ${G}$.

${G}$ is self twin (self dual) if ${\overline{G}}$ is identical
to ${G}$. We denote by $STP$ the set of self twin $P$ games.

Our starting point is that an alternative and more friendly description of the twin relationship may be obtained making recourse to a specific property of the vectorial representations of $P$ games.

Usually constant sum homogeneous weighted majority games,
including $P$ games, are described by their classical minimal
homogeneous representation $(q; \w)$ where
$\w=\{w_1,w_2,\ldots,w_n\}$ is the ordered $(w_1=1, w_i\leq
w_{i+1}$ for any $i$) vector of individual minimal homogeneous
integer weights and $q$ the minimal winning quota given by $q=\cfrac{1+\sum_{i\in N} w_i}{2}$ and satisfying $q=\sum_{i\in S} w_i$, for any coalition $S$ of the so called minimal winning set\footnote{Note that $N$ is the
``grand'' coalition of all players.}.

As shown in sect. \ref{Sect:Pgamesalt}, alternative unequivocal representations of a $P$ game may be given either through its binary representation or through its type representation.
In particular, the binary representation is the ordered $n$ dimensional vector $\b$ with
components $b_i=1$ if $w_i>w_{i-1}$, $b_i=0$ otherwise
$(w_i=w_{i-1})$. Note that the condition $w_0=0$, which gives zero weight to the fictitious player ($i=0$) representing the set of dummies, implies $b_1=1$.   The type representation is
the ordered $h$ $(2\leq h\leq n-2)$ dimensional vector
$\x=(x_1,x_2,\ldots,x_h)$ whose component
$x_t$ is the number of players of type $t$ (sharing the same
common type weight $w_t>w_{t-1}$) in the game.

Yet, some of the information embedded in such
representations is redundant; indeed, the first two and the
last two components of $\b$, like the last component
of $\x$, are the same in any $P$ game. Then, what
really matters in a $P$ game is the, $n-4$ dimensional, free
binary representation $_f\b=(b_3,\ldots,b_{n-3})$, or
respectively the, $h-1$ dimensional, free type representation
$_f\x=(x_1,x_2,\ldots,x_{h-1})$.

In terms of these free representations, $STP$ games are (as
shown in Pressacco-Plazzotta, 2013, \cite{PrePla13}, sect. 4.2,
p. 9) the subset of $P$ games characterized by bilateral
symmetry. Exploiting bilateral symmetry it is straightforward
to obtain closed form formulae for the cardinality of $STP$
games for feasible combinations of $(n,h)$ and to show that a
properly $MPT$, $\Gamma$, with entries $\Gamma(m,k)$, describes
the cardinality of $STP$ games with $n=m+4$ players and $h=k+2$
types.

Indeed we will show that the rule which
governs the evolution with $m$ of the $MPT$ mimics
the one holding in the classical Pascal triangle (i.e. the ``Pascal equation''), except for internal
entries corresponding to combinations of even $m$ and odd $k$. In detail, for such combinations is it $\Gamma(m,k)=\Gamma(m-1,k)-\Gamma(m-1,k-1)$, so as the difference
rather than the sum of the adjacent entries of the previous row
is to be computed (obviously this modification has a feedback
on the whole triangle).

On the way we describe also the
regularities of other related $MPT$: $\Delta$, which
describes the cardinality (still as a function of the couple
$m,k$) of the $P$ non self twin games, and $\Theta=\Delta/2$ the cardinality of couples of (non identical) twins.

Besides these counting
results, we found also that the whole set of $STP$
may be described in the form of a genealogical tree, derived
through a constructive procedure. In this approach successive
generations of $STP$ are associated to the sequence
of $m$ values, and each game may be seen as a parent which
gives birth to one child or two children (depending on the even
or odd character of the parent generation), $STP$ games of
the next generation. The breeding rules, defined in terms of
the free type representation, are simple and generations invariant,
albeit parity specific.

A key role in the evolution of the genealogical tree is played by the pivot components of the subset
of $STP$ games, whose $_f\x$ representation has an odd number of
components (shortly OSTP games). We show that the structure of the pivot's set is described by other ``regular'' triangles.

The plan of the paper is as follows. Section
\ref{sect:classicP} recalls basic concepts of majority games, resumes relevant results on $P$ games\footnote{Some of these results are directly given by Isbell, while others are derived as straightforward consequences
of his approach, including the role of the classic Pascal
triangle in describing the cardinality of $P$ games with $n$
players and $h$ types.} and introduces the alternative binary and type representations, which play a key role in the subsequent sections. Section \ref{sect:TwinAndST} is devoted to define bilateral symmetry seen as a particular aspect of the symmetry (or twin) relationship on $P$ games.
In Section \ref{sect:selftwin} closed form
expressions of the number of $STP$ games for different feasible
combinations of the parameters $n,h$ or $m,k$ are obtained.
Section \ref{sect:MPTinST}, the core of the paper, suggests and
discusses an interpretation in terms of $MPT$ of such formulas.
In Section \ref{sect:NewApproachST} we build a genealogical tree
of $STP$ games, through the application of simple evolutionary
rules that may be translated in high speed algorithms. Section
\ref{sect:Pivot} introduces the pivot triangles and give
results on their main evolutionary properties. Section
\ref{sect:example} gives some examples and conclusions
follow in Section \ref{sect:concl}.

\section{Definitions and main results on $P$ games}\label{sect:classicP}

\subsection{Basics of weighted majority games}\label{sect:classicp1}

As usual $N=(1,\ldots,n)$ denotes the set of all (non dummy) players.
\begin{defi}
A coalition $S$ is a subset of the set $N$. A game $G$ in coalitional function form is defined by the coalitional function of the game, that is the real function $v: \mathcal{P}(N) \rightarrow \mathbb{R}$.
\end{defi}
\begin{defi}
A game in coalitional function form is simple if its $v$ function has values in $\{0,1\}$. A coalition $S$ is winning if $v(S)=1$, losing if $v(S)=0$.
\end{defi}
\begin{defi}
A simple game is constant sum if, for any $S \in \mathcal{P}(N)$, $v(S)+v(\widetilde{S})=1$.
\end{defi}
\begin{defi}
A coalition $S$ is said to be minimal winning if $v(S)=1$ and, for any $T\varsubsetneq S$, $v(T)=0$. The set of minimal winning coalitions is denoted by $WM$. A player $i$ who does not belong to any minimal winning coalition is said dummy; at the other extreme, it is a dictator if it is $v(i)=1$.
We will consider games free of dictator and dummies.
\end{defi}
\begin{defi}
A simple, weighted majority game is described by a
representation $(q;\w)$ where $\w$ is a vector
$(w_1,w_2,\ldots,w_n)$ of positive weights,
$q>\cfrac{1}{2}\cdot w(N)=\cfrac{1}{2}\cdot\sum_{i\in N}w_i$ is
the winning quota and $v(S)=1\Leftrightarrow w(S)=\sum_{i\in
S}w_i\geq q$.
\end{defi}
\begin{defi}
A representation $(q;\w)$ of a weighted majority game is homogeneous if $w(S)=q$ for any $S$ of $WM$.
\end{defi}
\begin{defi}
Homogeneous weighted majority games are games for which (at least) a homogeneous representation exists.
\end{defi}

Consider now the class of all simple, constant sum $n$ person homogeneous weighted majority games\footnote{At the origins of game
theory, homogeneous weighted majority games (h.w.m.g.) have been
introduced in \cite{VNM47} by Von Neumann-Morgenstern and have
been studied mainly under the constant sum condition.
Subsequent treatments in the absence of the constant sum
condition (with deadlocks) may be found e.g. in \cite{Ost87} by
Ostmann, who gave the proof that any h.w.m.g. (including non
constant sum ones) has a unique minimal homogeneous
representation, and in \cite{Ros87}. Generally speaking, the
homogeneous minimal representation is to be thought in a
broader sense but hereafter the restrictive application
concerning the constant sum case is used.}.
\begin{risu}
All games of such a class admit a minimal homogeneous representation, that is a (homogeneous) representation $(q;\w)$ such that all weights are integers, there are players with (minimum) weight 1 and $q=\cfrac{w(N)+1}{2}$, which implies for any $S \in WM$, $w(S)-w(\widetilde{S})=1$.
\end{risu}

\begin{rem}
Hereafter we will suppose that in such a representation the vector $\w$ is ordered according to the convention $w_1=1, w_i\leq w_{i+1}$ for any $i$.
\end{rem}

\begin{risu}
The cardinality of the $WM$ set of our class may be either greater or equal (but not lower) than $n$. See \cite{Isb56}, p. 185.
\end{risu}

\subsection{Parsimoniuos games and their alternative representations}\label{Sect:Pgamesalt}

\begin{defi}
We call Parsimoniuos games (hereafter $P$ games) the subset of constant sum homogeneous weighted majority games characterized by the parsimony property to have for any given number $n$ of non dummy players in the game the smallest number, i.e. exactly $n$, of minimal winning coalitions.
\end{defi}

The nice properties of this class of games have been studied in a ground breaking paper (Isbell, 1956, \cite{Isb56}) going back to the early stage of game theory development.
In particular, there it has been found  that a $P$ game is univocally described also by an alternative binary representation.
\begin{defi}
Given a $P$ game $G$ with minimal homogeneous representation $(q;\w)$, its binary representation is the ordered (according to the non decreasing weight labelling) $n$ dimensional vector $\b$ with components\footnote{The initial condition $w_0=0$ which gives weight zero to the fictitious player labelled 0, representing the set of all dummy players, would give $b_1=1$, coherently with the general rule.}: $b_1=1$ and
for $i=2,\ldots,n$, $b_i=1$ if $w_i>w_{i-1}$, $b_i=0$
otherwise, i.e if $w_i=w_{i-1}$.
\end{defi}

\begin{risu}\label{Res:ris3}
The following constraints (\cite{Isb56}, p. 185) hold, in any
$P$ game, on the first two and on the last two components of
$\b$:
$b_1=1$, $b_2=0$, $b_{n-1}=0$, $b_n=1$.
Apart from these constraints, the choice of the other components of the binary representation is fully free.
\end{risu}

\begin{defi}
The free binary representation of a $P$ game with binary
representation $\b$ is the ordered $n-4$ dimensional vector of
the internal components of $\b$. We denote such a vector by
$_f\b= (b_3,\ldots,b_{n-2})$.
\end{defi}

\begin{risu}
There is a one-to-one correspondence between the set of all
feasible choices of $_f\b$ and the set of $P$ games; hence,
putting (for any $n>3$)\footnote{We will consider here $P$
games with $n>3$. Indeed, for $n=3$ there is a unique $P$ game
in which all players share the same weight, while all $P$ games
with $n>3$ have at least two types of players.} $m=n-4$, there
are exactly $2^m$ $P$ games with $n=m+4$ (non dummy) players.
\end{risu}

Another alternative representation of a $P$ game may be obtained from the following considerations.
Given the minimal homogeneous representation of the game, let us divide the players in $h$ non overlapping subsets, according to their weight. For any $t=1,\ldots,h$ let $w_t$ denote the type weight (i.e. the common individual weight of all players of the group labelled $t$). Types too are labelled according to the (now) strictly increasing weight convention, i.e. $w_1=1$, $w_{t-1}<w_t$ for any $t$.

\begin{defi}
The (complete) type representation of a $P$ game is the ordered h dimensional vector $\x=(x_1,x_2,\ldots,x_h)$ whose component $x_t$ is the number of players of group $t$ in the game.
\end{defi}
Note that
\begin{risu}
$(n-2)\geq h\geq 2$
\end{risu}

Indeed there are as many types in a game as the number of 1's components in its binary representation and we know that there are at least two zeros as well as at least two 1's. Moreover it turns out that in any $P$ game the following constraints hold on $\x$ :

\begin{risu}
$$x_h=1$$
$$x_1\geq 2$$
$$x_{h-1}\geq 2$$
\end{risu}

\begin{proof} Immediate consequence of the constraints given by Result \ref{Res:ris3}.
\end{proof}

\begin{rem}
For $h=2$, $h-1=1$ and $x_1=n-1$.
\end{rem}

Keeping account of the surely binding constraint $(x_h =1)$, the really relevant (indeed full) information on a $P$ game is embedded in the first $(h-1)$ components of $\x$.

\begin{defi}
The truncated or free\footnote{Note that the other constraints
on $x_1$ and $x_{h-1}$ may leave some freedom in the choice of
the related components or be fully binding depending on the
values of the couple $(n,h)$.} type vector
$_f\x=(x_1,x_2,\ldots,x_{h-1})$ is the $h-1$ dimensional vector
obtained by deletion of the last (obliged) component of $\x$.
\end{defi}

To summarize: besides by their minimal homogeneous representation  $(q;\w)$ $P$ games are unequivocally described either by $_f\b$ or by $_f\x$.

Indeed the connection between $_f\x$ and $(q;\w)$ is given by the following recursive relation on type weights:
\begin{equation}\label{Fnct:recursive}
\begin{split}
w_0&=0\\
w_1&=1\\
w_t&=x_{t-1}\cdot w_{t-1}+ w_{t-2}, \quad t=2,\ldots, h-1\\
w_h&=(x_{h-1}-1)\cdot w_{h-1}+w_{h-2}\\
\end{split}
\end{equation}

\paragraph{Example.} Just to give an example consider the nine person $P$ game with five types and free type structure  $_f\x=(2,2,1,3)$.
There are two players of the first type with individual (and
type) weight $w_1=w_2=1$; the next two players (labelled 3 and
4) of type 2 with type weight $w_2=x_1\cdot w_1+ w_0=2$, so
that $w_3=w_4=2$; one player of type 3 with type weight
$w_3=x_2\cdot w_2+ w_1=5$ hence individual weight $w_5=5$;
three players of type 4 and type weight $w_4=x_3\cdot w_3+
w_2=7$ and hence individual weights $w_6=w_7=w_8=7$ and finally
one top player (type 5) whose type weight is given by
$w_5=(x_4-1)\cdot w_4+w_3=19$. The sequence $\w$ of individual
weights in the minimal homogeneous representation is then
$(1,1,2,2,5,7,7,7,19)$ and the corresponding minimal winning
quota is $q=26$.

After that, even if it has not been explicitly stated in the Isbell paper, the same argument used there to derive the number of $n$ person $P$ games may be applied to obtain immediately that:

\begin{risu}
The number of $P$ games with $n=m+4$ players and $h=k+2$ types is given, for any $k=0,\ldots,m,$ by:
\begin{equation}\label{Fnct:numbP}
C_{m,k}=\cfrac{m!}{k!(m-k)!}
\end{equation}
\end{risu}

Hence putting $m=0,1,\ldots,n-4$ on rows and $k=0,\ldots,m$ on
columns, the classic Pascal triangle gives the number of
different $P$ games with $n=m+4$ players and $h=k+2$ types.
Keeping account of this coincidence, we will denote by $C$
(with components $C(m,k)=C_{m,k}=\frac{m!}{k!(m-k)!})$ this
triangle.
In the next sections we wish to extend these old results to
obtain a $MPT$, $\Gamma$, which gives for any combination of
$(m,k)$, the number $\Gamma(m,k)$ of $STP$ games with $n=m+4$
players and $h=k+2$ types.

\section{Bilateral symmetry in $P$
games}\label{sect:TwinAndST}

In this section we introduce bilateral symmetry in $P$
games\footnote{Elsewhere (Pressacco-Plazzotta, 2013,
\cite{PrePla13}, sect. 4, Theorem 4.2) we show that bilateral
symmetry is the exact counterpart of self duality in $P$ games
a property suggested, as said in the introduction, by Isbell
(\cite{Isb56}, pp. 185-186) on the basis of special character
of the incidence matrices of such games.}.
\begin{defi}\label{risu:risu1}
The free binary representation $_f\b$ of a $P$ game $G$ is bilaterally symmetric or self symmetric if, for any $i=3,\ldots,n-2$:
$$
b_i=b_{n+1-i}
$$
\end{defi}

\begin{prop} \label{prop:3.1}
Bilateral symmetry of $_f\b$ is equivalent to bilateral symmetry of $\b$.
\end{prop}

\begin{proof}
Immediate keeping account that, in any $P$ game,
$b_1=b_n=1$ and $b_2=b_{n-1}=0$.
\end{proof}

\begin{defi}
The free type representation $_f\x$ of a $P$ game $G$ is bilaterally symmetric or self symmetric if,
for any $t=1,\ldots,h-1$:
$$x_t=x_{h-t}$$
\end{defi}

The following result links bilateral symmetry of the $_f\b$ and $_f\x$ representations of a game:
\begin{risu}
Let $_f\b$ and $_f\x$ free binary and free type representations of a game $G$; $_f\b$ is bilaterally symmetric if and only if $_f\x$ is bilaterally symmetric.
\end{risu}

Proof is given in Appendix \ref{App:A}.

\begin{defi}
A $P$ game is bilaterally symmetric (shortly is $BSP$) if its free representations (binary or type) are bilaterally symmetric.
\end{defi}

In the next sections we will often use the terminology ``self twin'' games in place of ``bilaterally symmetric'' games. More general symmetry relationships in $P$ games come along the following lines.

\begin{defi}
A couple ${G}$, ${\overline{G}}$ of $P$ games are each other
symmetric if their free binary representations $_f{\b}$ and
$_f{\overline{\b}}$ are each other symmetric, i.e if, for any
$i=3,\ldots,n-2$:
$$
{\overline{b_i}}={b_{n+1-i}}
$$
\end{defi}

\begin{rem}
This property too may be given equivalently in terms of symmetry of the free type representations.
\end{rem}

\begin{defi}
A couple ${G}$, ${\overline{G}}$ of $P$ games are each other
symmetric if their free type representations $_f{\x}$ and
$_f{\overline{\x}}$ are each other symmetric, i.e if for any
$t=1,\ldots,h-1$:
$$
{\overline{x_t}}={x_{h-t}}
$$
\end{defi}

\begin{defi}
A couple ${G}$, ${\overline{G}}$ of each other symmetric $P$
games are also called twins.
\end{defi}

\begin{defi}
A game ${G}$, whose symmetric is just itself, is said self
symmetric or self twin.
\end{defi}

\begin{rem}
Obviously $G$ is self symmetric if and only if $G$ is
bilaterally symmetric. Each $P$ game ${G}$ has just one
symmetric ${\overline{G}}$; a given couple of symmetric games
surely have free type representations with the same number
$(h-1)$ of components.
\end{rem}

In conclusion self twin games are a particular subset of $P$ games characterized by bilateral symmetry of their free representations; in the next section we will derive the cardinality of such subsets as a function of $n$ or respectively of the couple $(n,h)$.

\section{The cardinality of self twin $P$
games}\label{sect:selftwin}

In this section we resume the results on the cardinality of
self twin or bilaterally symmetric games. First of all, using the notation $\Gamma(m)=\Gamma(n-4)$ with $m=0,1,\ldots,n-4$ for the number
of $STP$ games with $n=m+4$ players, it is:
\begin{risu}\label{risu:risu4.1}

\begin{subequations}\label{Fnct:cardin1}
\begin{align}
\Gamma(m)&= 2^{m/2}=2^{(n-4)/2}, \quad \text{m (or n) even}\label{3first}\\
\Gamma(m)&=2^{(m+1)/2}=2^{(n-3)/2}, \quad \text{m (or n) odd}\label{3second}
\end{align}
\end{subequations}
\end{risu}

\begin{rem}\label{rem:rem1}
For $n$ even, this is just the squared root of the overall
number of $n$ person $P$ games. Then for $n=4$ ($m=0$), the only $P$
game is self twin; for $n=6$ ($m=2$) there are four $P$ games, two of
which self twin; for $n=8$ ($m=4$), four self twin out of sixteen $P$
games and so on; the number of self twin games multiplies by
two as we pass from an even $n$ to the next odd $n+1$.
\end{rem}

\begin{proof}
For $n$ even, we have a free binary choice for each component
of the first half of the free binary representation vector,
that is $m/2=(n-4)/2$ components (while of course the second
half is fully constrained by symmetry); for $n$ odd, we can
freely choose the first $(m-1)/2$ components and the pivot,
that is $(m+1)/2=(n-3)/2$.
\end{proof}

As for the number $\Gamma(m,k)=\Gamma(n-4,h-2)$ of  $STP$ games with $n$ players
and $h$ types\footnote{Obviously, $n$ and $m$ and, respectively, $h$ and $k$ have the same parity.}, it is:
\begin{risu}\label{risu:risu4.2}

\begin{subequations}\label{Fnct:cardin2}
\begin{align}
\Gamma(m,k)&=C_{m/2,k/2} \quad \text{m even, k even}\label{4first}\\
\Gamma(m,k)&=0 \quad \text{m even, k odd}\label{4second}\\
\Gamma(m,k)&=C_{(m-1)/2,k/2} \quad \text{m odd, k even}\label{4third}\\
\Gamma(m,k)&=C_{(m-1)/2,(k-1)/2} \quad \text{m odd, k odd}\label{4forth}
\end{align}
\end{subequations}
\end{risu}

\begin{proof}
a) Both $m$ and $k$ even. There is a free choice of
      the $k/2=(h-2)/2$ components equal to 1 among the
      $m/2=(n-4)/2$ places in the first half of the free
      binary representation; by bilateral symmetry there is no free
      choice in the second half.

b) If $m$ even and $k$
      odd. No bilateral symmetry at all may be obtained.

c) If $m$ (and
      then $n-4$) is odd and $k$ even, bilateral symmetry may be obtained
      only if the pivot (median) component of the free
      binary representation is 0. Then, we are left with
      $m-1=n-5$ (free components) and $k=h-2$ (free elements
      equal to 1) both even, and we are brought back,
      keeping account of the symmetry constraint, to the
      result sub a).

d) If both $m$ and $k$ are odd, bilateral symmetry
      may be obtained only if the pivot component of the
      free binary vector is 1. Then both $m-1=n-5$ (free
      components) and $k-1=h-3$ (free elements equal to 1)
      are even and we are brought back again to the result
      sub a).
\end{proof}

\section{$MPT$ in self twin games}\label{sect:MPTinST}

The results of the previous section give rise to a modified
version $\Gamma$ of the classic Pascal triangle, as may be
perceived by Table 1
(on rows $m=n-4$, starting from $m=0$ or $n=4$; on columns
$k=h-2$, going from $k=0$ or $h=2$ to $k=m$ or $h=n-2$), where
you find the number of $STP$ games for combinations of $m$
and $k$.

\begin{table}[H]
  \begin{center}\label{Tabella:Tab1}\scriptsize
    \begin{tabular}{rr|rrrrrrrrrr}
    \multicolumn{12}{c}{$k$}\\
     & & 0    & 1    & 2    & 3   & 4 & 5 & 6 & 7 & 8 & ...\\
    \hline
    \multirow{9}{0.1cm}{\rotatebox{90}{~\parbox{0.5cm}{$m$}~}}   &  0 & 1    &     &     &    &  & & & & \\
       &  1 & 1 & 1 \\
       &  2 & 1 & 0 & 1 \\
       &  3 & 1 & 1 & 1 & 1\\
    &  4 & 1 & 0 & 2 & 0 & 1\\
       &  5 & 1 & 1 & 2 & 2 & 1 & 1\\
       &  6 & 1 & 0 & 3 & 0 & 3 & 0 & 1\\
       &  7 & 1 & 1 & 3 & 3 & 3 & 3 & 1 & 1\\
       &  8 & 1 & 0 & 4 & 0 & 6 & 0 & 4 & 0 & 1\\
       &... & ... & ... & ... & ... & ... & ... & ...& ... &...&... \\
    \hline
  \end{tabular}
\end{center}
\caption{Triangle $\Gamma$ with entries $\Gamma(m,k)$, the
cardinality of $STP$ games with $n=m+4$ players and $h=k+2$
types.}
\end{table}

We wish to underline that the definition of (modified) Pascal
triangle for $\Gamma$ is not unfounded. Indeed the entries of $\Gamma$
are far from being merely a collection of non negative integers
summing up for any row to a power of 2; on the contrary, they
are all binomial coefficients (or zero) whose evolution obeys
almost wholly (with only a slight modification) the basic
recurrent combinatorial rule (Pascal equation) of the classic Pascal triangle
$C(m,k)=C(m-1,k)+C(m-1,k-1)$.

Hereafter the evolutionary rules of the $\Gamma$ triangle.
\begin{risu}\label{risu:risu5.1}

    For any $m$ and for $k=0$ or $k=m$:
      \begin{equation}\label{eqn:new9}
      \Gamma(m,0)=\Gamma(m,m)=1
      \end{equation}

      For $k=1,\ldots,m-1$ and $(m+k)$ even:
          \begin{equation}\label{eqn:new10}
      \Gamma(m,k)=\Gamma(m-1,k)+\Gamma(m-1,k-1)
      \end{equation}

      For $k=1,\ldots,m-1$ and $(m+k)$ odd:
    \begin{equation}\label{eqn:ex54}
    \Gamma(m,k)= \Gamma(m-1,k)-\Gamma(m-1,k-1)
    \end{equation}
\end{risu}

\begin{rem}
In the case $m$ odd, $k$ even, it is $(m-1)$ even and $(k-1)$
odd, and by Formula \eqref{4second} $\Gamma(m-1,k-1)=0$. Hence,
in this case, the Pascal equation would not be really
challenged. Then the rows with odd $m$ go on following the
classic Pascal recurrence equation, which needs to be adapted only for
the case $m$ even, $k$ odd.
\end{rem}

The straightforward but tedious proof of Formulas
\eqref{eqn:new10} and \eqref{eqn:ex54} is given in Appendix
\ref{App:B}. Let us shortly comment here the result
\eqref{eqn:new9}.

For any given $n$ (or equivalently $m=n-4$), there is just one
$P$ game with only two types ($h=k+2=2)$. Obviously it
corresponds to the case in which the free binary representation
of the game is the null vector. This implies that the free type
representation has only one element $x_1=n-1$. Hence, this $P$
game is surely characterized by bilateral symmetry (i.e. it is
self twin). Then $\Gamma(m,0)=1$. We recall that these games,
well known in literature, are called Apex games (see
\cite{Hor73}, 1973).

Symmetrically, for any given $n$, there is just one $P$ game
with $h=n-2$ types. It corresponds to the case in which the
free binary representation of the game is the unit vector. This
implies that the free type representation has $n-3$ elements
(the largest possible, given $n$), with $x_1=x_{n-3}=2$ and all
the other elements $x_t=1$. Hence, also this $P$ game is
characterized by bilateral symmetry (i.e. it is self twin).
Then $\Gamma(m,m)=1$. These games have been introduced by
Isbell in (\cite{Isb56}, p. 185) who underlined that the
sequence of the individual weights is the Fibonacci one, with
the only slight modification that the weight of the vicetop
player is repeated twice or, more formally, denoting by $f_i$ the
sequence of Fibonacci numbers: $w_i=f_i$ for $i=1,\ldots,n-2$,
$w_{n-1}=f_{n-2}$, $w_n=f_{n-1}$. Hence such games may be
called Fibonacci games. For example, the sequence of individual weights of the
Fibonacci game with $n=10$ players is
$\{1,1,2,3,5,8,13,21,21,34\}$.

\vspace{0.5cm}

If we are interested in counting the number of non self twin,
we find another $MPT$, $\Delta$, whose entries (see Table 2)
are obviously defined for any $m,k$ combination by:
\begin{equation}
\Delta(m,k)=C(m,k)-\Gamma(m,k)
\end{equation}

\begin{table}[H]\label{Tabella:Tab2}
  \begin{center}\scriptsize
    \begin{tabular}{rr|rrrrrrrrrr}
    \multicolumn{12}{c}{$k$}\\
     & & 0    & 1    & 2    & 3   & 4 & 5 & 6 & 7 & 8 & ...\\
    \hline
    \multirow{9}{0.1cm}{\rotatebox{90}{~\parbox{0.5cm}{$m$}~}}   &  0 & 0 & & & &  & & & & \\
       &  1 & 0 & 0 \\
       &  2 & 0 & 2 & 0 \\
       &  3 & 0 & 2 & 2 & 0\\
       &  4 & 0 & 4 & 4 & 4 & 0\\
       &  5 & 0 & 4 & 8 & 8 & 4 & 0\\
       &  6 & 0 & 6 & 12 & 20 & 12 & 6 & 0\\
       &  7 & 0 & 6 & 18 & 32 & 32 & 18 & 6 & 0\\
       &  8 & 0 & 8 & 24 & 56 & 64 & 56 & 24 & 8 & 0\\
       &... & ... & ... & ... & ... & ... & ... & ...& ... &...&... \\
    \hline
  \end{tabular}
\end{center}
\caption{Cardinality $\Delta(m,k)=C(m,k)-\Gamma(m,k)$ of non
$STP$ games for any combination of $m,k$.}
\end{table}

It turns out that also $\Delta$ is a $MPT$, whose evolutionary
rules are as follows:
\begin{risu}\label{risu:risu5.2}

        For any $m$ and for $k=0$ or $k=m$:
      \begin{equation}
      \Delta(m,0)=\Delta(m,m)=0
      \end{equation}
    Proof obvious.

      For $k=1,\ldots,m-1$ and $(m+k)$ even:
      \begin{equation}
      \Delta(m,k)=\Delta(m-1,k)+\Delta(m-1,k-1)
      \end{equation}
    Proof.
    \begin{equation}
    \begin{split}
    \Delta(m,k)&=C(m,k)-\Gamma(m,k)=\\
    &=C(m-1,k)+C(m-1,k-1)-(\Gamma(m-1,k)+\Gamma(m-1,k-1))\\
    &=C(m-1,k)-\Gamma(m-1,k)+C(m-1,k-1)-\Gamma(m-1,k-1)\\
    &=\Delta(m-1,k)+\Delta(m-1,k-1)
    \end{split}
    \end{equation}

    For $k=1,\ldots,m-1$ and $(m+k)$ odd:
    \begin{equation}
    \Delta(m,k) = \Delta(m-1,k-1)+\Delta(m-1,k)+2\Gamma(m-1,k-1)
    \end{equation}
    Proof.
    \begin{equation}
    \begin{split}
    \Delta(m,k)&=C(m,k)-\Gamma(m,k)=\\
    &=C(m-1,k)+C(m-1,k-1)-(\Gamma(m-1,k)-\Gamma(m-1,k-1))\\
    &=C(m-1,k)+C(m-1,k-1)-(\Gamma(m-1,k)+\Gamma(m-1,k-1)-2\Gamma(m-1,k-1))\\
    &=C(m-1,k)-\Gamma(m-1,k)+C(m-1,k-1)-\Gamma(m-1,k-1)+2\Gamma(m-1,k-1)\\
    &=\Delta(m-1,k)+\Delta(m-1,k-1)+2\Gamma(m-1,k-1)
    \end{split}
    \end{equation}
\end{risu}

\begin{rem}
Recall that for odd $m$, even $k$ it is $\Gamma(m-1,k-1)=0$.
Once more this implies that in this case the classic Pascal
equation is not challenged, so that it needs to be really
updated only for the combinations even $m$, odd $k$.
\end{rem}

Among other things it turns out that all non zero components of
$\Delta$ are even; this allows to introduce another, associated to $\Delta$,
modified Pascal triangle $\Theta$, whose entries are defined for any
$m,k$ by:
\begin{equation}
    \Theta(m,k)=\Delta(m,k)/2
\end{equation}
which gives the number of pairs of non identical twins to be
added to the identical twins to obtain the overall number of P
games for any combination of $m,k$. See Table 3.

\begin{table}[H]\label{Tabella:Tab3}
  \begin{center}\scriptsize
    \begin{tabular}{rr|rrrrrrrrrr}
    \multicolumn{12}{c}{$k$}\\
     & & 0    & 1    & 2    & 3   & 4 & 5 & 6 & 7 & 8 & ...\\
    \hline
    \multirow{9}{0.1cm}{\rotatebox{90}{~\parbox{0.5cm}{$m$}~}}   &  0 & 0 & & & &  & & & & \\
       &  1 & 0 & 0 \\
       &  2 & 0 & 1 & 0 \\
       &  3 & 0 & 1 & 1 & 0\\
       &  4 & 0 & 2 & 2 & 2 & 0\\
       &  5 & 0 & 2 & 4 & 4 & 2 & 0\\
       &  6 & 0 & 3 & 6 & 10 & 6 & 3 & 0\\
       &  7 & 0 & 3 & 9 & 16 & 16 & 9 & 3 & 0\\
       &  8 & 0 & 4 & 12 & 28 & 32 & 28 & 12 & 4 & 0\\
       &... & ... & ... & ... & ... & ... & ... & ...& ... &...&... \\
    \hline
  \end{tabular}
\end{center}
\caption{Cardinality $\Theta(m,k)=\Delta(m,k)/2$ of pairs of
non identical twins for any combination of $m,k$.}
\end{table}

See section \ref{sect:example} for a presentation of some
examples.

\section{A constructive approach to self twin
games}\label{sect:NewApproachST}

In this section we propound a constructive evolutionary
approach to $STP$ games in which successive generations of
such games (each generation being associated to a value of $m$)
are generated by $STP$ games of the previous generation
through the application of simple rules. The base of this
approach is the free type representation of a $P$ game
$_f\x=(x_1, x_2,\ldots,x_{h-1}$) introduced in the final
part of section \ref{Sect:Pgamesalt}, with  $STP$ games
characterized by bilateral symmetry of their free type
representation. We discuss this approach starting from the
following picture \ref{Fig:GenealTree}, which gives the initial
evolution of $STP$ games (described by the above
representation), starting from the seed, the ``Adamo'' game
with representation (3), that is the only $P$ game with 4 players (which is $STP$).

\begin{figure}[H]
\begin{center}
\includegraphics[type=eps,ext=.eps,read=.eps,width=11cm]{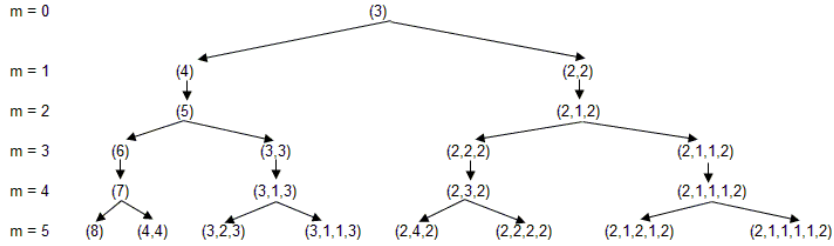} \qquad
\caption{Genealogical tree of self twin games (rows $m$
from 0 to 5 or $n$ from 4 to 9) expressed by their free type
representation.} \label{Fig:GenealTree}
\end{center}
\end{figure}

Let us give here rather informally the evolutionary rules that
characterize the tree.

\begin{prop}\label{prop:0}
The seed of the genealogical tree (the generation zero game) is
the $STP$ ``Adamo'' four person game, whose free type
representation is 3.
\end{prop}

\begin{defi} A $STP$ game with an odd (respectively an even)
number of components in its free type representation $_f\x$ is
called odd (even) self twin parsimonious, or shortly $OSTP$
($ESTP$) game.
\end{defi}

\begin{prop}\label{prop:1}
All STP games of zero or even generation (m zero or even), are
OSTP i.e. they have $h-1$ odd.
\end{prop}

\begin{prop}\label{prop:2}
Half of the STP games of any odd generation (m odd) are
OSTP (and obviously the other half ESTP).
\end{prop}

\begin{defi} The central component $\x_{h/2}$ of the free type
representation of an OSTP is called pivot of the representation
(or shortly of the game).
\end{defi}

\begin{prop}\label{prop:3}
The pivot of an OSTP of any (zero or) even generation is an
odd number.
\end{prop}

\begin{prop}\label{prop:4}
The pivot of an OSTP of any odd generation is an even
number.
\end{prop}

\begin{prop}\label{prop:5}
An OSTP $G$ of any even generation, with pivot $\x_{h/2}$ gives
birth to a couple of children STP games of the next (odd)
generation, let us denote them by $G^{1}$ and $G^{2}$. $G^{1}$
is OSTP and $G^{2}$ is ESTP.
\end{prop}

\begin{prop}\label{prop:6}
The child $G^{1}$ replicates the representation of the parent
$G$ except for the pivot which is now (coherently with prop.
\ref{prop:4}) surely an even number $\x^{1}_{h/2}=\x_{h/2}+1$.
\end{prop}

\begin{prop}\label{prop:7}
The child $G^{2}$ is obtained by an almost full replication of
the brother $G^{1}$ but (coherently with prop. \ref{prop:2})
splitting the pivot $\x_{h/2}+1$ into a couple of equal central
components $(\x_{h/2}+1)/2$.
\end{prop}

\begin{prop}\label{prop:8}
An OSTP $G^{1}$ of any odd generation, with even pivot
$\x^{1}_{h/2}=\x_{h/2}+1$ gives birth to a unique child OSTP
game $G^{11}$ of the next (even) generation. The child
replicates the structure of the father (and of the grandfather
also) except for the pivot which is now (coherently with prop.
\ref{prop:4}) surely an odd number
$\x^{11}_{h/2}=\x^{1}_{h/2}+1=\x_{h/2}+2$.
\end{prop}

\begin{prop}\label{prop:9}
An ESTP $G^{2}$ of any odd generation gives birth to a unique
child OSTP game $G^{21}$ of the next (even) generation. The
child replicates the structure of the father but with the
insertion of the pivot 1 in the previous even representation of
$G^{2}$.
\end{prop}

\begin{prop}\label{prop:10}
All the games built according to these rules are by
construction bilaterally symmetric (hence STP).
\end{prop}

\begin{prop}\label{prop:11}
The number of STP (half OSTP, half ESTP) of any odd
generation is twice the number of the STP (all OSTP) of the
previous generation.
\end{prop}

\begin{prop}\label{prop:12}
The number of STP (all OSTP) of any even generation $m$ is the
same of the STP of the previous odd generation $m-1$ and twice
the number of STP of the previous even $m-2$ generation.
\end{prop}

\begin{prop}\label{prop:13}

Prop. \ref{prop:0}, \ref{prop:11} and \ref{prop:12} jointly
grant that the STP of our genealogical tree exhaust the
cardinality of STP games for any value of $n$ (or $m$). Check
it through a simple comparison with the values of $\Gamma(m)$
given by Result \ref{risu:risu4.1}.
\end{prop}

We may conclude that:

\begin{prop}\label{prop:14}
The genealogical tree introduced in this section fully
and unequivocally describes all ST (bilaterally symmetric) Parsimonious games.
\end{prop}

In this section we found that the pivot
components of $OSTP$ played a key role in the evolutionary rules
of the genealogical tree. In the next section we will see that
other triangles, no more of the $MPT$ type but nevertheless
characterized by significant regularities, are able to describe
the evolutionary rule behind the pivot behaviour.

\section{The pivot triangles}\label{sect:Pivot}

A pivot triangle is a triangle of numbers whose rows correspond to generations according to their $m$ value; in each row we find the sequences, in increasing order, of all feasible values of the pivots of the $OSTP$ games for the related generation. Columns are labelled by natural numbers $c=1,2,\ldots$. The value of the pivot in row $m$ and column $c$ is $x_{m,c}$, while $r_{m,c}$ is the number of $OSTP$ games of $m$ generation whose pivot is $x_{m,c}$ (shortly repetitions of $x_{m,c}$).

We will distinguish hereafter even triangles, associated to even generations ($m$ zero or even), and odd triangles associated to odd values of $m$.

\subsection{The pivot triangles of even generations}

Let us consider in this section only the even generations of
the tree. For any of the initial even generations, we list here
(Table 4) 
the values $x_{m,c}$ of the pivots involved, as
well as (in brackets) the number $r_{m,c}$ of repetitions of
each pivot's value.

\begin{table}[H]\label{Tabella:Tab71}
  \begin{center}\scriptsize
    \begin{tabular}{rr|rrrrrrr}
    \multicolumn{9}{c}{$c$}\\
     & & 1& 2 & 3 & 4 & 5 & 6 & 7 \\
    \hline
      &0  &  3(1)  & \\
      &2  &  1(1)  & 5(1) \\
      &4  &  1(2)  & 3(1)  & 7(1)  \\
     {\rotatebox{90}{~\parbox{0.01cm}{$m$}~}} &6  &  1(4)  & 3(2)  & 5(1) & 9(1)\\
      &8  &  1(8)  & 3(4)  & 5(2) & 7(1) & 11(1)\\
      &10 &  1(16) & 3(8)  & 5(4) & 7(2) & 9(1) & 13(1)\\
      &12 &  1(32) & 3(16) & 5(8) & 7(4) & 9(2) & 11(1) & 15(1)\\
  \end{tabular}
\end{center}\caption{Even generation pivots triangle}
\end{table}

To understand the table note that the evolutionary rules of the
genealogical tree imply that going from one even generation, $m$,
to the next even, $m+2$, each pivot $x_{m,c}$ gives rise to a couple
of grandchildren pivots, whose values are, respectively, $x_{m+2,c+1}=x_{m,c}+2$ and $x_{m+2,1}=1$. Hence,
if in the $m$ generation the odd pivot $x_{m,c}$ is repeated $r_{m,c}$
times, the odd pivot $x_{m,c}+2$ is repeated $r_{m,c}=r_{m+2,c+1}$
times in the following $m+2$ generation.

On the other side, Formula \eqref{3first} implies that
the odd pivot $x_{m,1}=1$ is repeated $r_{m+2,1}=2^{m/2}$ times in the
$m+2$ generation\footnote{Note that this must be equal to the sum of
numbers within brackets in the $m$ generation}.

These rules,
jointly with the initial condition that in generation 0 there
is the unique pivot $x_{0,1}=3$, give a straightforward explanation
of the whole table. Let us enter now in some detail.

\begin{prop}
For any row $m$ the number of columns, that is of different pivots of $m$ generation, is $\cfrac{m}{2}+1$.
\end{prop}
Indeed, it is the cardinality of the set of all positive odd integers not greater than $m+3$, except for $m+1$.

The pivots' values are as follows:
\begin{risu}
\begin{subequations}\label{Fnct:pivots}
\begin{align}
x_{0,1}=x_{0,(m/2)+1}=3 \quad \text{there is just one pivot}\label{Pivfirst}\\
x_{m,c}=2c-1\quad \text{for $m$ even and $c=1,\ldots,(m/2)$} \label{Pivsecond}\\
x_{m,c}=2c+1\quad\text{for $m$ even and $c=(m/2)+1$}  \label{Pivthird}
\end{align}
\end{subequations}
\end{risu}

Put $Z(m)=\sum_c x_{m,c}$, the sum of all pivots' values (without repetitions) of the $m$ generation. Then, it is:
\begin{risu}
\begin{subequations}\label{Fnct:SumPiv}
\begin{align}
    Z(m)=2+\Big(\cfrac{m}{2}+1\Big)^2 \text{or}\label{SumFirst}\\
    Y(m)=Z(m)-2=\Big(\cfrac{m}{2}+1\Big)^2\label{SumSecond}
\end{align}
\end{subequations}

\end{risu}

\begin{proof}
Subtracting 2 to the last entry of each row (the other entries unchanged), we find exactly the well known triangle of squared numbers, i.e. the triangle having on the rows the sequences of all odd integers up to $m+1$, whose sum of entries (on each row) is $\Big(\cfrac{m}{2}+1\Big)^2$, that is the sequence of squared natural numbers $(1,4,9,16,\ldots)$. See Table 5.
\end{proof}

\begin{table}[H]\label{Tabella:Tab73}
  \begin{center}\scriptsize
    \begin{tabular}{rr|rrrrrrr|rrr}
    & & & & & $c$ & & & &\\
     &&1&2 &3 & 4 & 5 & 6 & 7 &$Y(m)$\\
    \hline
      &0&1  &  & &&&&&$1$\\
      &2&1  &  3 &&&&&&$4$ \\
      &4&1  &  3 & 5 &&&&&$9$   \\
      {\rotatebox{90}{~\parbox{0.01cm}{$m$}~}}&6&1  &  3 & 5 & 7 &&&&$16$\\
      &8&1  &  3 & 5 & 7 & 9 &&&$25$\\
      &10&1  &  3 & 5 & 7 & 9 & 11 &&$36$\\
      &12&1  &  3 & 5 & 7 & 9 & 11 & 13 & $49$\\
      \hline
  \end{tabular}
\end{center}\caption{Squared numbers triangle.}
\end{table}

The sum of the row $m+2$ is obtained adding $m+3$ to the sum of the (non bracketed) entries of the previous row $m$ (see Table 4):
\begin{risu}
\begin{equation}
    Z(m+2)=Z(m)+(m+3)
\end{equation}

\end{risu}

\begin{proof}
Immediate by Formula \eqref{SumFirst}. An enlightening pivotal explanation is the following: we know that each row (label $m$) has exactly $1+(m/2)$ columns. Going from the $m$ generation to the following, each of the $(m+2)/2$ pivot with value $x$ generates a pivot of value $x+2$; moreover, the pivot 1 is to be added (here we do not care the repetitions). The global increment is then $1+(2\cdot(m+2)/2)=1+(m+2)=m+3$.
\end{proof}

As for the repetitions, we have:
\begin{risu}
\begin{subequations}\label{Fnct:repetitions}
\begin{align}
r_{m,c}&=2^{(m/2)-c}\quad \text{for $m$ even and $c=1,\ldots,(m/2)$} \label{Repfirst}\\
r_{m,(m/2)+1}&=1\quad\text{for $m$ (zero or even) and $c=(m/2)+1$}  \label{Repsecond}
\end{align}
\end{subequations}
\end{risu}

This implies:
\begin{risu}
\begin{subequations}\label{Fnct:repetitions2}
\begin{align}
r_{m,1}&=\sum_c r_{m-2,c}\label{Rep2first}\\
r_{m,c}&=2\cdot r_{m-1,c}\quad\text{for $m$ even and $c=1,\ldots,(m/2)-1$} \label{Rep2second}\\
r_{m,c}&=1\quad\text{for $m$ even and $c=(m/2),(m/2)+1$} \label{Rep2third}
\end{align}
\end{subequations}
\end{risu}

It turns out that the highest rightward diagonal (corresponding to the first row $m=0$) has constant entries equal to 1, while those starting from the other rows ($m$ even), have constant entries $2^{(m/2)-1}$. See Table 6.

\begin{table}[H]\label{Tabella:Tab74}
  \begin{center}\scriptsize
    \begin{tabular}{rr|rrrrrrrrrrr}
        \multicolumn{9}{c}{$c$}\\
     & & 1& 2 & 3 & 4 & 5 & 6 & 7 \\
    \hline
     &  0  &  (1) & \\
     &  2  &  (1) & (1) \\
     &  4  &  (2) & (1) & (1)  \\
          {\rotatebox{90}{~\parbox{0.01cm}{$m$}~}} &6  &  (4)  & (2)  & (1) & (1)\\
     &  8  &  (8) & (4) & (2) & (1) & (1)\\
     & 10  &  (16) & (8) & (4) & (2) & (1) & (1)\\
     & 12  &  (32) & (16) & (8) & (4) & (2) & (1) & (1)\\
  \end{tabular}
\end{center}\caption{Pivot repetitions triangle}
\end{table}

Finally, as for the sum $\Phi(m)$ of all pivots value on each row (keeping account of the repetitions), it is for any $m$ (zero or even):
\begin{risu}
\begin{equation}\label{Eqn:Phi}
\Phi(m)=\sum_c x_{m,c}\cdot r_{m,c}=2^{m/2}\cdot 3
\end{equation}
\end{risu}

\begin{proof}
Immediate by induction. It is true for $m=0$; let us check that if it is true for any even $m$, it is true for the next even $(m+2)$.
It is:
\begin{equation}
\begin{split}
\Phi(m+2)&=\sum_c x_{m+2,c}\cdot r_{m+2,c}=\sum_c (x_{m,c}+3)\cdot r_{m,c}=\\
&=\sum_c x_{m,c}\cdot r_{m,c}+3\cdot\sum_c r_{m,c}=\Phi(m)+3\cdot 2^{m/2}=\\
&=2^{m/2}\cdot 3+3\cdot 2^{m/2}=2\cdot2^{m/2}\cdot 3=2^{(m+2)/2}\cdot 3\\
\end{split}
\end{equation}
\end{proof}

\subsection{The pivot triangles of odd generations}

Let us consider now the behaviour of the odd generations
($m=1,3,\ldots$) and in particular the pivot elements of the
subset with odd representation. It turns out that it mimics
(mutatis mutandis) the rule driving the even generation one;
this fact may be perceived by the following table:
\begin{table}[H]\label{Tabella:Tab75}
  \begin{center}\scriptsize
    \begin{tabular}{rr|rrrrr}
    \multicolumn{7}{c}{$c$}\\
    & &  1& 2 & 3 & 4 & 5 \\
    \hline
      & 1  &  4(1) & \\
      & 3  &  2(1) & 6(1) \\
      {\rotatebox{90}{~\parbox{0.01cm}{$m$}~}}&5  &  2(2) & 4(1) & 8(1)  \\
      & 7  &  2(4) & 4(2) & 6(1) & 10(1)\\
      & 9  &  2(8) & 4(4) & 6(2) & 8(1) & 12(1)\\
  \end{tabular}
\end{center}\caption{Odd generation pivot triangle}
\end{table}

To understand the behaviour of the new triangle requires
nothing but to adapt all properties.

Let $x^{'}_{m^{'},c}$ and $r^{'}_{m^{'},c}$ be the entries of the odd triangle, with $m^{'}=m+1$ ($m$ even). It is immediate to check that, for any $m^{'},c$, we have:
\begin{prop}
\begin{equation}\label{Eqn:modd}
x^{'}_{m^{'},c}=x_{m,c}+1
\end{equation}
\begin{equation}\label{Eqn:rodd}
r^{'}_{m^{'},c}=r_{m,c}
\end{equation}
\end{prop}

Exploiting Formulas \eqref{Eqn:modd} and \eqref{Eqn:rodd} it is
easy to obtain the counterpart of all results given in the
previous subsection. Here, we provide only the analogous of
Formula \eqref{Eqn:Phi}. Putting $$\Psi(m^{'})=\sum_c
x^{'}_{m^{'},c}\cdot r_{m^{'},c}$$ it is:
\begin{risu}
\begin{equation}\label{Eqn:Psi}
\Psi(m^{'})=2^{(m^{'}-1)/2}\cdot 4
\end{equation}
\end{risu}
\begin{proof}
It is:
\begin{equation}
\begin{split}
\Psi(m^{'})&=\sum_c x^{'}_{m^{'},c}\cdot r_{m^{'},c}=\sum_c (x_{m,c}+1)\cdot r_{m,c}=\\
&=\sum_c x_{m,c}\cdot r_{m,c}+\sum_c r_{m,c}=\Phi(m)+ 2^{m/2}=\\
&=2^{m/2}\cdot 3+2^{m/2}=4\cdot 2^{m/2}=4\cdot 2^{(m^{'}-1)/2}\\
\end{split}
\end{equation}
\end{proof}

\section{Examples}\label{sect:example}

\paragraph{Example 8.1.} The first example keeps $m=4$, that is $n=8$. There are
$2^4=16$ $P$ games, 4 of which are self twin, while the other
12 may be seen as 6 pairs of non identical twins. Table 3
reveals that those pairs may be divided in 3 groups with respectively three, four or five types. Each group contains 2 pairs of non identical twins. Here the complete list of the 16 $P$ games, grouped by
their $h$ value and expressed through their free type
representation (in parenthesis) with $h-1$ components.

\begin{table}[H]
  \begin{center}\scriptsize
    \begin{tabular}{rrl}
      $h=2$: & (7) & self twin\\
    \hline
      $h=3$: & (3,4) and (4,3) & first pair of twins\\
             & (2,5) and (5,2) & second pair of twins\\
    \hline
      $h=4$: & (2,3,2) & first self twin\\
             & (3,1,3) & second self twin\\
             & (2,1,4) and (4,1,2) & first pair of twins\\
             & (2,2,3) and (3,2,2) & second pair of twins\\
    \hline
      $h=5$: & (2,1,2,2) and (2,2,1,2) & first pair of twins\\
             & (2,1,1,3) and (3,1,1,2) & second pair of twins\\
     \hline
      $h=6$: & (2,1,1,1,2) & self twin\\
    \hline
  \end{tabular}
\end{center}
\end{table}

The corresponding rows of our $MPT$ are ($m=4$, $k$ from 0 to
4):
  $$C: 1,4,6,4,1$$
  $$\Gamma: 1,0,2,0,1$$
  $$\Delta: 0,4,4,4,0$$
  $$\Theta: 0,2,2,2,0$$

The minimal homogeneous representations are:
\begin{table}[H]
  \begin{center}\scriptsize
    \begin{tabular}{rrr}
      $h=2$: & (7;1,1,1,1,1,1,1,6) & \\
    \hline
      $h=3$: & (13;1,1,1,3,3,3,3,10) & (13;1,1,1,1,4,4,4,9)\\
             & (11;1,1,2,2,2,2,2,9) & (11;1,1,1,1,1,5,5,6)\\
    \hline
      $h=4$: & (16;1,1,2,2,2,7,7,9) & \\
             & (15;1,1,1,3,4,4,4,11) & \\
             & (14;1,1,2,3,3,3,3,11) & (14;1,1,1,1,4,5,5,9)\\
             & (17;1,1,2,2,5,5,5,12) & (17;1,1,1,3,3,7,7,10)\\
    \hline
      $h=5$: & (19;1,1,2,3,3,8,8,11) & (19;1,1,2,2,5,7,7,12)\\
             & (18;1,1,2,3,5,5,5,13) & (18;1,1,1,3,4,7,7,11)\\
     \hline
      $h=6$: & (21;1,1,2,3,5,8,8,13) & \\
    \hline
  \end{tabular}
\end{center}
\end{table}

It is interesting to observe that (as proved in
\cite{PrePla13},Theorem T5.1, sect. 5.1), all pairs of twins
have the same minimal winning quota $q$ (of course pair
dependent).

\paragraph{Example 8.2}

The second example keeps $n=9$ ($m=5$). There are 32 $P$ games,
8 of which self twin, while the other 24 may be divided in 12
pairs of twins.

Here the complete list of the 32 $P$ games, grouped by their
$h$ value:

\begin{table}[H]
  \begin{center}\scriptsize
    \begin{tabular}{rrl}
      $h=2$: & (8) & self twin\\
    \hline
      $h=3$: & (4,4) & self twin\\
             & (3,5) and (5,3) & first pair of twins\\
             & (2,6) and (6,2) & second pair of twins\\
    \hline
      $h=4$: & (2,4,2) & first self twin\\
             & (3,2,3) & second self twin\\
             & (2,1,5) and (5,1,2) & first pair of twins\\
             & (2,2,4) and (4,2,2) & second pair of twins\\
             & (2,3,3) and (3,3,2) & third pair of twins\\
             & (3,1,4) and (4,1,3) & forth pair of twins\\
    \hline
      $h=5$: & (2,2,2,2) & first self twin\\
             & (3,1,1,3) & second self twin\\
             & (2,1,1,4) and (4,1,1,2) & first pair of twins\\
             & (2,1,2,3) and (3,2,1,2) & second pair of twins\\
             & (2,1,3,2) and (2,3,1,2) & third pair of twins\\
             & (3,1,2,2) and (2,2,1,3) & forth pair of twins\\
     \hline
      $h=6$: & (2,1,2,1,2) & self twin\\
             & (2,1,1,2,2) and (2,2,1,1,2) & first pair of twins\\
             & (2,1,1,1,3) and (3,1,1,1,2) & second pair of twins\\
    \hline
      $h=7$: & (2,1,1,1,1,2) & self twin\\
    \hline
  \end{tabular}
\end{center}
\end{table}

The corresponding rows of our $MPT$ are ($m=5$, $k$ from 0 to
5):
  $$C: 1,5,10,10,5,1$$
  $$\Gamma: 1,1,2,2,1,1$$
  $$\Delta: 0,4,8,8,4,0$$
  $$\Theta: 0,2,4,4,2,0$$

The minimal homogeneous representations are:
\begin{table}[H]
  \begin{center}\scriptsize
    \begin{tabular}{rrl}
      $h=2$: & (8;1,1,1,1,1,1,1,1,7) & \\
    \hline
      $h=3$: & (17;1,1,1,1,4,4,4,4,13) & \\
             & (16;1,1,1,3,3,3,3,3,13) & (16;1,1,1,1,1,5,5,5,11)\\
             & (13;1,1,2,2,2,2,2,2,11) & (13;1,1,1,1,1,1,6,6,7)\\
    \hline
      $h=4$: & (20;1,1,2,2,2,2,9,9,11) & \\
             & (24;1,1,1,3,3,7,7,7,17) & \\
             & (17;1,1,2,3,3,3,3,3,14) & (17;1,1,1,1,1,5,6,6,11)\\
             & (22;1,1,2,2,5,5,5,5,17) & (22;1,1,1,1,4,4,9,9,13)\\
             & (23;1,1,2,2,2,7,7,7,16) & (23;1,1,1,3,3,3,10,10,13)\\
             & (19;1,1,1,3,4,4,4,4,15) & (19;1,1,1,1,4,5,5,5,14)\\
    \hline
      $h=5$: & (29;1,1,2,2,5,5,12,12,17) & \\
             & (25;1,1,1,3,4,7,7,7,18) & \\
             & (23;1,1,2,3,5,5,5,5,18) & (23;1,1,1,1,4,5,9,9,14)\\
             & (27,1,1,2,3,3,8,8,8,19) & (27;11,1,1,3,3,7,10,10,17)\\
             & (25;1,1,2,3,3,3,11,11,14) & (25;1,1,2,2,2,7,9,9,16)\\
             & (26;1,1,1,3,4,4,11,11,15) & (26;1,1,2,2,5,7,7,7,19)\\
     \hline
      $h=6$: & (30;1,1,2,3,3,8,11,11,19) & \\
             & (31;1,1,2,3,5,5,13,13,18) & (31;1,1,2,2,5,7,12,12,19)\\
             & (29;1,1,2,3,5,8,8,8,21) & (29;1,1,1,3,4,7,11,11,18)\\
    \hline
      $h=7$: & (34;1,1,2,3,5,8,13,13,21) & \\
    \hline
  \end{tabular}
\end{center}
\end{table}

\section{Conclusions}\label{sect:concl}

In this paper we discuss the role of modified Pascal triangles
in describing the cardinality of self twin (bilaterally
symmetric) Parsimonious games for any combination of the
relevant parameters $m,k$ associated respectively to the number
$n=m+4$ of players and $h=k+2$ of types in the game. In detail,
we show that the entries of the modified triangles follow
almost wholly (with a slight modification) the evolutionary
rule embedded in the basic combinatorial relation (Pascal
equation) which gives any binomial coefficient as the sum of
two adjacent coefficient of the previous row of the classic
Pascal triangle. In addition we also provide a genealogical
tree of self twin games, in which each game of a given
generation (corresponding to a value of $m$) is able to give
birth to one or two (depending on the parity of $m$) children
self twin games of the next generation. The breeding rules,
defined in terms of the free type representation, are, given
the parity, invariant across generations. They are
quite simple and may be translated in an high speed pen and
pencil constructive procedure to obtain all self twin games for
small enough values of $n$; obviously a simple computational routine
produces the set of all self twin games also for large values
of $n$ with the only constraint of the computational power
(recall that the number of self twin games explodes at the
rhythm of the square root of $2^{n-4}$).

The analysis of the genealogical tree revealed that a key role in
its evolution is played by the pivot components of the subset
of self twin Parsimonious games whose free type representation has an odd number of
components. We found that other triangles
(pivot triangles) describe the structure of the pivot's set and
give a synthesis of their evolutionary pattern.

We are aware that our paper is wholly
theoretical so we leave to subsequent research to look for
practical applications to hard or social sciences. Indeed, we
hope that such results could follow given the prominent role
played by bilateral symmetry in many fields of the life of the
universe\footnote{There is a huge literature concerning
symmetry in hard sciences; besides the references given in the introduction, let us recall here some prominent
sentences: ``Symmetry is one idea by which man through the ages
has tried to comprehend and create order, beauty and
perfection'' (\cite{We52}, p. 5); ``Symmetry considerations
dominate modern fundamental physics both in quantum theory and
in relativity'' (\cite{BrCa03}, p. ix preface); ``Symmetry
plays an essential role in science'' (\cite{Ros96}, editor
foreword); ``Bilateral symmetry is a hallmark of Bilateralia''
(\cite{Fi03}, abstract); ``Understanding the origin and
evolution of bilateral organism requires an understanding of
how bilateral symmetry develops, starting from a single cell''
(\cite{Wie12}, abstract).}.

\appendix
\section{Appendix: proof of the proposition \ref{prop:3.1}}\label{App:A}

Define $\I(\b)=(I_1,I_2,\ldots,I_h)$ as the vector of ordinal labellings of 1's of $\b$.

\vspace{0.2cm}

For example, if $\b=(1,0,1,0,1,1,0,1,0,1)$, $\I(\b)=(1,3,5,6,8,10)$.

\vspace{0.2cm}

It follows that, for $t=1,\ldots,h-1$, it is:
$$x_t=I_{t+1}-I_{t} \quad \text{and } \quad x_h=1$$

Conversely, for any $t>2$, it is:
$$
I_{t}=1+\sum_{j=1}^{t-1}x_j\quad \text{and } \quad I_1=1
$$

In the example: $x_1=3-1=2$, $x_2=5-3=2$, $x_3=6-5=1$, $x_4=8-6=2$, $x_5=10-8=2$ and $x_6=1$, or $_f\x=(2,2,1,2,2)$.

\vspace{0.2cm}

Hypothesis: $_f\b$ bilaterally symmetric, hence $\b$ bilaterally symmetric. See Proposition \ref{prop:3.1}.

We show that this implies that also $_f\x$ is bilaterally symmetric.
\begin{prop}\label{prop:A1}
Bilateral symmetry of $\b$ is equivalent to the following property of $\I(\b)$:
\end{prop}

\begin{pr}\label{pr:A1}
For any couple of symmetric subscripts $t$ and $h+1-t$ (sum of subscripts $h+1$) it is:
$$
I_t+I_{h+1-t}=n+1
$$
\end{pr}

In the example, it is:
$$
I_1+I_6=1+10=I_2+I_5=3+8=I_3+I_4=5+6=10+1
$$

Considering now the difference $(x_t-x_{h-t})$, it is for any
$t=1,\ldots,h-1$:

$$
x_t-x_{h-t}=(I_{t+1}-I_{t})-(I_{h+1-t}-I_{h-t})=(I_{t+1}+I_{h-t})-(I_t+I_{h+1-t})=(n+1)-(n+1)=0
$$

The last but one equality is a consequence of Property
\ref{pr:A1}, keeping account that the sum of indices of the
terms in both brackets is $h+1$. Hence $ x_t-x_{h-t}=0$, which
means that $_f\x$ is bilaterally symmetric.

\vspace{0.2cm}
Conversely, hypothesis: $_f\x$ bilaterally symmetric. We show that this implies that also $_f\b$ is bilaterally symmetric.

By hypothesis, for any $t=1,\ldots,h-1$, it is:
$$
x_t=x_{h-t}
$$

Substitution gives:
$$
I_{t+1}-I_{t}=I_{h+1-t}-I_{h-t}
$$
equivalent to:
$$
I_{t+1}+I_{h-t}=I_{t}+I_{h+1-t}
$$
The sum of subscripts in both sides of the equation is $h+1$, thus Property \ref{pr:A1} is satisfied and, by proposition \ref{prop:A1} this implies bilateral symmetry of $\b$.

\section{Appendix: proofs of evolutionary rules of the $MPT$
$\Gamma$}\label{App:B}

\begin{proof} Case $m$ even, $k$ even. Recalling Formula \eqref{4first}, it is:
\begin{equation}
      \Gamma(m,k)=C(m/2,k/2)
\end{equation}
Now, denoting by $m^{'}=m-1$ and $k^{'}=k-1$ both odd and
keeping account of Formula \eqref{4third}:
\begin{equation}
      \Gamma(m-1,k)=\Gamma(m^{'},k)=C\Big((m^{'}-1)/2,k/2\Big)=C\Big((m/2)-1,k/2\Big)
\end{equation}
while, applying Formula \eqref{4forth}:
\begin{equation}
      \Gamma(m-1,k-1)=\Gamma(m^{'},k^{'})=C\Big((m^{'}-1)/2,(k^{'}-1)/2\Big)=C\Big((m/2)-1,(k/2)-1\Big)
\end{equation}
Being
$$
C\Big(m/2,k/2\Big)=C\Big((m/2)-1,(k/2)\Big)+C\Big((m/2)-1,(k/2)-1\Big)
$$
we found that Formula \eqref{eqn:new10} is satisfied.
\end{proof}

\begin{proof} Case $m$ odd, $k$ odd. Recalling Formula \eqref{4forth}, it is:
\begin{equation}
      \Gamma(m,k)=C\Big((m-1)/2,(k-1)/2\Big)
\end{equation}
In turn, for Formula \eqref{4second}:
\begin{equation}
      \Gamma(m-1,k)=0
\end{equation}
while, applying Formula \eqref{4first}:
\begin{equation}
      \Gamma(m-1,k-1)=\Gamma(m^{'},k^{'})=C\Big(m^{'}/2,k^{'}/2\Big)=C\Big((m-1)/2,(k-1)/2\Big)
\end{equation}
so that we found again that Formula \eqref{eqn:new10} is satisfied.
\end{proof}

\begin{proof} Case $m$ odd, $k$ even. Recalling Formula \eqref{4third}, it is:
\begin{equation}
      \Gamma(m,k)=C\Big((m-1)/2,k/2\Big)
\end{equation}
Now, keeping account of Formula  \eqref{4first}:
\begin{equation}
      \Gamma(m-1,k)=\Gamma(m^{'},k)=C\Big(m^{'}/2,k/2\Big)=C\Big((m-1)/2,k/2\Big)
\end{equation}
on the other side, applying Formula \eqref{4second} ($m-1$
even, $k-1$ odd):
\begin{equation}
      \Gamma(m-1,k-1)=0
\end{equation}
So that Formula \eqref{eqn:ex54} is satisfied.
\end{proof}

\begin{proof} Case $m$ even, $k$ odd. Recalling Formula \eqref{4second}, it is:
\begin{equation}
      \Gamma(m,k)=0
\end{equation}
and let us check that:
\begin{equation}
      \Gamma(m-1,k)=\Gamma(m-1,k-1)
\end{equation}
Indeed
\begin{equation}
      \Gamma(m-1,k)=\Gamma(m^{'},k)=C\Big((m^{'}-1)/2,(k-1)/2\Big)
\end{equation}
while
\begin{equation}
      \Gamma(m-1,k-1)=\Gamma(m^{'},k^{'})=C\Big((m^{'}-1)/2,k^{'}/2\Big)=C\Big((m^{'}-1)/2,(k-1)/2\Big)
\end{equation}
Again Formula \eqref{eqn:ex54} is satisfied.
\end{proof}

\end{document}